\documentclass{ws-aejm}
\usepackage{amssymb,amsfonts,amscd,latexsym}

\newtheorem{prop}{Proposition}[section]
\newtheorem{coroll}[prop]{Corollary}
\newtheorem{thm}[prop]{Theorem}
\newtheorem{exam}[prop]{Example}
\newtheorem{rem}[prop]{Remark}
\newtheorem{rems}[prop]{Remarks}
\newtheorem{problem}{Problem}

\numberwithin{equation}{section}

\newcommand{\NN}{{\mathbb N}}

\newcommand{\CC}{{\mathbb C}}
\newcommand{\mnc}{{M_n(\CC)}}

\newcommand{\glimma}{{\hbox{\it Glimm}\kern.05em(A)}}
\newcommand{\glimmb}{{\hbox{\it Glimm}\kern.05em(B)}}
\newcommand{\prima}{{\hbox{\it Prim}\kern.05em(A)}}
\newcommand{\primb}{{\hbox{\it Prim}\kern.05em(B)}}
\newcommand{\irra}{{\hbox{\it Irr}\kern.05em(A)}}
\newcommand{\irrb}{{\hbox{\it Irr}\kern.05em(B)}}
\newcommand{\speca}{{\hbox{\it Spec}\kern.05em(A)}}
\newcommand{\specb}{{\hbox{\it Spec}\kern.05em(B)}}
\newcommand{\specza}{{\hbox{\it Spec}\kern.05em(Z(A))}}
\newcommand{\rest}[2]{{{#1}_{\kern-.5pt|{#2}}}}  

\newcommand{\pfi}{\varphi}
\newcommand{\sep}[1]{{\mathcal{S}({#1})}}
\newcommand{\rad}[1]{{\textup{rad}({#1})}}
\def\C*{{\sl C*}-algebra}
\def\Cs*{{\sl C*}-subalgebra}
\def\AF/{{\sl AF}-algebra}
\def\AFs/{{\sl AF}-algebras}
\def\vN/{{von Neumann algebra}}
\def\II1f/{{\textup{II}}$_1$\kern2pt{factor}}
\def\spbddop/{{spectrally bounded operator}}
\def\spbdd/{{spectrally bounded}}
\def\spisom/{{spectral isometry}}
\def\spisoms/{{spectral isometries}}
\def\SR/{{\sl SR}-algebra}
\def\Tone/{{\sl T}$_1$}
\def\Ttwo/{{\sl T}$_2$}

\begin{document}

\setcounter{page}{1}

\markboth{Martin Mathieu}{A Collection of Problems on Spectrally Bounded Operators}

\catchline{?}{?}{2009}{}{}

\title{A COLLECTION OF PROBLEMS ON\\ SPECTRALLY BOUNDED OPERATORS%
            \footnote{This paper is an expanded version of a talk given at the conference \textsl{Jordan Structures: Nonassociative Analysis and Geometry\/}
                                  on 6 September 2008 at Queen Mary College, London.}}
\author{MARTIN MATHIEU}
\address{Department of Pure Mathematics, Queen's University Belfast,\\ Belfast BT7 1NN, Northern Ireland\\
                   m.m@qub.ac.uk}

\maketitle
\begin{history}
\end{history}

\begin{dedication}
\textit{Dedicated to Professor Rainer Nagel on the occasion of his retirement.}
\end{dedication}

\begin{abstract}
We discuss several open problems on spectrally bounded operators, some new, some old, adding in a few new insights.
\end{abstract}



\section{Introduction}\label{sec:intro}

\noindent
Let $A$ be a unital complex Banach algebra. A linear mapping $T\colon E\to B$ from a subspace $E\subseteq A$ into another unital complex
Banach algebra $B$ is called \textit{spectrally bounded\/} if there is a constant $M\geq0$ such that $r(Tx)\leq M\mkern.6mu r(x)$ for all $x\in E$. Here,
and in what follows, $r(x)$ stands for the spectral radius of a Banach algebra element~$x$.

This concept evolved in Banach algebra theory, and especially automatic continuity, over time in the 1970's and 1980's but the terminology was only
introduced in~\cite{MM1994a}, together with its companions \textit{spectrally infinitesimal}: $M=0$; \textit{spectrally contractive}: $M=1$; and
\textit{spectrally isometric}: $r(Tx)=r(x)$ for all~$x$. It follows from~\cite{Aup82}, see also \cite{AupMM}, Lemma~A, that the separating
space of every surjective \spbddop/ $T$ on a closed subspace~$E$ is contained in the radical of~$B$; thus $T$ is bounded if $B$ is semisimple.
This was used by Aupetit in~\cite{Aup82}
to give a new proof of Johnson's uniqueness-of-the-complete-norm-topology theorem; cf.\ also~\cite{Ran89}.
It was further exploited in~\cite{AupMM} to investigate continuity properties of Lie epimorphisms.

A systematic study of \spbddop/s was begun in~\cite{Schick}, with its main results published in~\cite{MMS1} and~\cite{MMS2}.
Since then, the interest in the topic has steadily grown and by now there is a sizeable literature on a variety of aspects.
The present paper's aim is to discuss several of the open problems on \spbddop/s along with some new results that
should make it clear why these questions are natural and important.

Section~\ref{sec:auto-cont} is devoted to a recapitulation of probably the most important and deepest problem on
\spbddop/s, the non-commutative Singer--Wermer conjecture. No substantial progress seems to have been made on this over
the last few years. A standard method to reduce a more sophisticated problem in Banach algebra theory to a simpler one
is to use quotients. When dealing with operators between Banach algebras, we, of course, need invariant ideals to perform this.
In Section~\ref{sec:quotients} we discuss the interplay between properties of a (\spbdd/) operator and the operator it induces
on suitable quotient Banach algebras in a fairly general setting and recover recent results for operators preserving the
essential spectral radius on $B(H)$ in~\cite{BBS08}.

The identity element plays a distinguished role for \spbddop/s. Suppose that $T\colon A\to B$ is a surjective \spisom/
between unital \C*s $A$ and~$B$. Since $T$ restricted to the centre $Z(A)$ of $A$ is a *-isomorphism onto the centre
$Z(B)$ of $B$ (\cite{MM05}, Proposition~2.3), $T1$ is a central unitary in~$B$. Therefore, whenever we study a \spisom/ $T$,
we can without loss of generality assume that $T$ is unital (i.e., $T1=1$); see the proof of \cite{LMM}, Corollary~2.6.
For a general \spbddop/, the situation changes and for this reason we shall pay attention in Section~\ref{sec:value-at-one}
to the relevance of the value~$T1$. As it turns out, if the domain algebra is sufficiently `infinite', then $T1$ must be a central invertible
element in~$B$. As a consequence, we are able to extend the main results in~\cite{MMS2} and~\cite{LMM} to the non-unital setting
(Theorem~\ref{thm:nonunital-spbdd} below). The special case $A=B(H)$, $B=B(K)$ for Hilbert spaces $H$ and $K$ was studied
previously in~\cite{CuiHou}.

In~\cite{MM05}, we propose a non-selfadjoint version of Kadison's classical theorem stating that every unital surjective isometry
between unital \C*s is a Jordan *-isomorphism. In many cases, this conjecture (see Problem~\ref{prob:spiom-conjecture} below)
has been confirmed, however almost always under some strong assumption of `infiniteness' on the domain \C*.
In Section~\ref{sec:spec-isoms}, we employ a reduction method developed in~\cite{MMSou} and~\cite{MMRud} to obtain a new result
of a similar ilk (Theorem~\ref{thm:no-tracial-states}). More importantly, perhaps, we propose a new route based on a description
of \spbddop/s in the presence of a trace to solve this open question at least in the setting of \II1f/s (Corollary~\ref{cor:type-twoone}).

\section{Automatic Continuity}\label{sec:auto-cont}

\noindent
The \textit{separating space\/} $\sep T$ of an operator $T$ between normed spaces $E$ and $F$ is defined by
\[
\sep T=\{y\in F\mid y=\lim_nTx_n\text{ for some }(x_n)_{n\in\NN}\subseteq E,\ x_n\to0\}
\]
and measures the degree of discontinuity of~$T$. If $E$ and $F$ are Banach spaces then $T$ is bounded if and only if
$\sep T=\{0\}$, by the closed graph theorem. For a detailed discussion of the separating space see~\cite{Dal}.

It follows from \cite{Aup82}, see also \cite{Ran89} and~\cite{AupMM}, that $\sep T\cap TE$ consists of quasinilpotent
elements whenever $T\colon E\to B$ is a \spbddop/ on a subspace $E$ of a Banach algebra. In fact,
$r(Tx)\leq r(Tx-y)$ for each $x\in E$ and $y\in\sep T$. This, together with Zem\'anek's characterisation of the Jacobson
radical $\rad A$ of~$A$, entails the following result.

\begin{thm}[Aupetit]\label{thm:aupetit-aut-cont}
Let\/ $T\colon E\to B$ be a \spbddop/ on a closed subspace\/ $E$ of a Banach algebra onto a Banach algebra~$B$. Then\/ $\sep T\subseteq\rad B$.
In particular, if\/ $B$ is semisimple, then\/ $T$ is bounded.
\end{thm}

The above estimate on the spectral radii is used in \cite{AupMM} to describe the continuity of surjective Lie homomorphisms as follows.

\begin{thm}[Aupetit--Mathieu]\label{thm:aut-cont-lie}
The separating space\/ $\sep\theta$ of a Lie epimorphism\/ $\theta$ between two Banach
algebras\/ $A$ and\/ $B$ is contained in\/ $\mathcal{Z}(B)$, the centre modulo the radical.
\end{thm}
The centre modulo the radical, $\mathcal{Z}(B)$, is the pre-image of the centre of $\hat B=B/\rad B$ under the canonical map
$B\to\hat B$ and turns out to be the largest quasi\-nil\-potent Lie ideal of~$B$. As a consequence, it is invariant under every
Lie epimorphism and thus the above result yields the following automatic continuity statement (\cite{AupMM}, Corollary).
\begin{coroll}
Let\/ $\theta\colon A\to B$ be a Lie epimorphism between the Banach algebras\/ $A$ and~$B$. The induced Lie epimorphism\/
$\hat\theta\colon\hat A= A/\rad A\to\hat B$ between the Banach Lie algebras\/ $\hat A$ and\/ $\hat B$ is continuous.
\end{coroll}

The centre modulo the radical also plays an important role in the automatic continuity of derivations on Banach algebras.
It is an open question whether the separating space of each of the iterates $\delta^n$, $n\in\NN$ of a derivation 
$\delta\colon A\to A$ on a Banach algebra $A$
is contained in the radical of~$A$. If this is the case, the following long-standing problem would have a positive answer.

\begin{problem}[Noncommutative Singer--Wermer Conjecture]\label{prob:ncsw}
Does $[x,\delta x]\in\mathcal{Z}(A)$ for all $x\in A$ imply that $\delta A\subseteq\rad A$?
\end{problem}
Here, $[x,y]$ of course stands for the commutator $xy-yx$. If $\mathcal{Z}(A)$ is replaced by the proper centre $Z(A)$ of $A$,
then Problem~\ref{prob:ncsw} has an affirmative answer, due to a reduction technique developed in~\cite{MMRun} and Marc Thomas' famous
theorem~\cite{Thom}. The connection to \spbddop/s is provided by the next result~\cite{BresMM}.

\begin{thm}[Bre\v sar--Mathieu]
A derivation on a Banach algebra\/ $A$ is \spbdd/ if and only if it maps into\/ $\rad A$.
\end{thm}
Hence, Problem~\ref{prob:ncsw} is equivalent to the following question.

\medskip
\noindent{\bf Problem 1$^\prime$.}
Does $[x,\delta x]\in\mathcal{Z}(A)$ for all $x\in A$ imply that $\delta$ is \spbdd/?

\medskip
It appears that no progress on this question has been made in the last decade. For a fuller discussion see~\cite{MM1994a} and~\cite{Dal}.

\section{Quotients}\label{sec:quotients}

\noindent
In the previous section, we exploited the concept of \spbddop/s to study linear mappings satisfying some
additional algebraic conditions, such as Lie homomorphisms, derivations, etc. Another main direction of research
has been on the question which algebraic properties can be derived from the assumption of spectral boundedness.
Maybe the most prominent of these problems is the following one from~\cite{Kap}. By a \textit{Jordan epimorphism\/}
we will, of course, understand a surjective linear mapping preserving the Jordan product $x\circ y=\frac12(xy+yx)$.

\begin{problem}[Kaplansky]\label{prob:kaplansky}
Let $A$ and $B$ be semisimple unital Banach algebras. Suppose
$T\colon A\to B$ is unital, surjective and invertibility-preserving.
Is $T$ necessarily a Jordan epimorphism?
\end{problem}
Note that a Jordan epimorphism necessarily has all the properties assumed in Problem~\ref{prob:kaplansky} and hence
is bounded, by Theorem~\ref{thm:aupetit-aut-cont}. Many contributions on Kaplansky's problem have been made over the
past decades but so far it eludes a final answer. From Aupetit's substantial work on the question we only quote the following
beautiful result in~\cite{Aup00}; see also the references therein.
\begin{thm}[Aupetit]\label{thm:aupetit-main}
Let\/ $A$ and\/ $B$ be von Neumann algebras. Then every unital surjective invertibility-preserving
linear mapping\/ $T\colon A\to B$ is a Jordan epimorphism.
\end{thm}
To the best of our knowledge, Problem~\ref{prob:kaplansky} is still open even in the case when both domain and codomain
are \C*s; cf.\ also~\cite{Harr01}. The structural investigation of \spbddop/s is partly motivated by the question to what extent
the hypotheses in Kaplansky's  problem can be relaxed while retaining the same goal, to show that the mapping is a Jordan homomorphism.
The relevance of the value at~$1$ will be discussed in the next section. When considering the possibility of replacing
``invertibility-preserving'' by ``\spbdd/'' we must keep in mind that (a)~\textit{every\/} bounded linear mapping defined on a commutative
\C* is \spbdd/ (since spectral radius and norm coincide in spaces of the form $C(X)$); (b)~a~finite trace on a \C* is a spectral contraction, hence, already
on the $n\times n$ matrices there is a unital \spbddop/ onto~$\CC$ which is not a Jordan epimorphism.
Notwithstanding this there are satisfactory results in the setting of `very infinite' \C*s; see Sections~\ref{sec:value-at-one} and~\ref{sec:spec-isoms}.

Many examples are known illustrating the fact that no strong results can be expected for non-surjective \spbddop/s in general;
see, e.g.,~\cite{Harr01}.  We shall now study the situation when  a \spbddop/ is merely surjective `up to' or `modulo' an ideal.

Suppose $I\subseteq A$ and $J\subseteq B$ are proper closed ideals in the unital Banach algebras $A$ and~$B$, respectively.
We say that a linear mapping $T\colon A\to B$ is \textit{essentially \spbdd/\/} (more precisely, $I$-$J$-\textit{essentially \spbdd/})
if there is a constant $M\geq0$ such that
$r(Tx+J)\leq M\mkern.6mu r(x+I)$ for all $x\in A$.
We call $T$ \textit{surjective modulo\/} $J$ if, for every $y\in B$, there is $x\in A$ such such $y-Tx\in J$.
If $TI\subseteq J$, we can define the induced linear mapping $\hat T\colon A/I\to B/J$ by $\hat T(x+I)=Tx+J$, $x\in A$.
The following proposition relates the properties of $\hat T$ to those of~$T$.

\begin{prop}\label{prop:ess-spec-bdd}
Let\/ $A$ and\/ $B$ be unital Banach algebras. Suppose that\/ $I$ and\/ $J$ are proper closed ideals of\/ $A$ and\/ $B$, respectively,
such that\/ $B/J$ is semisimple. For a linear mapping\/ $T\colon A\to B$ the following conditions are equivalent.
\begin{enumerate}
\item[\textup{(a)}] $T$ is essentially spectrally bounded and surjective modulo~$J$;
\item[\textup{(b)}] $TI\subseteq J$ and\/ $\hat T$ is \spbdd/ and surjective.
\end{enumerate}
\end{prop}
\begin{proof}\quad
(b)${}\Rightarrow{}$(a)\ Under the assumption $TI\subseteq J$, $\hat T$ is a well-defined linear mapping which is \spbdd/ if and only
if $T$ is essentially \spbdd/, by definition. Moreover, $\hat T$ surjective precisely means that, for each $y\in B$,
$y+J=\hat T(x+I)=Tx+J$ for some $x\in A$; that is, $T$ is surjective modulo~$J$.

\smallskip\noindent
(a)${}\Rightarrow{}$(b)\ By the first paragraph in this proof, it suffices to show that $TI\subseteq J$.
Take $x\in I$. By hypothesis, for each $y\in B$, there is $x'\in A$ such that $y-Tx'\in J$. Let $\lambda\in\CC$. We have
\begin{equation*}
\begin{split}
r(\lambda\,(Tx+J)+y+J) &= r(\lambda\,Tx+y+J)\\
                                                 &= r(T(\lambda\,x+x')+J)\\
                                                 &\leq M\mkern.6mu r(\lambda\,x+x'+I)\\
                                                 &= M\mkern.6mu r(x'+I),
\end{split}
\end{equation*}
for some $M\geq0$. Consequently, the subharmonic function $\lambda\mapsto r(\lambda\,(Tx+J)+y+J)$
is bounded on $\CC$, hence constant. It follows that $r(Tx+J+y+J)=r(y+J)$ for all $y\in B$ from which we conclude
that $Tx+J\in\rad{B/J}$, by Zem\'anek's characterisation of the radical. As $B/J$ is semisimple, we obtain that $Tx\in J$ as desired.
\end{proof}
As an immediate consequence we have the following variant of the main result in~\cite{LMM}.
\begin{coroll}\label{cor:pure-infinite-quotient}
Let\/ $T\colon A\to B$ be a unital linear mapping from a unital purely infinite \C*\/ $A$ with real rank zero into a unital Banach algebra~$B$.
Let\/ $J\subseteq B$ be a proper closed ideal in\/ $B$ such that\/ $B/J$ is semisimple and suppose that, for some proper closed ideal\/
$I\subseteq A$ of\/ $A$, the
operator\/ $T$ is\/ $I$-$J$-essentially \spbdd/ and surjective modulo~$J$. Then\/ $T$ is a Jordan epimorphism modulo~$J$.
\end{coroll}
\begin{proof}
By Proposition~\ref{prop:ess-spec-bdd}, $TI\subseteq J$ and the induced mapping $\hat T\colon A/I\to B/J$ is unital, \spbdd/ and surjective.
Since $A/I$ is purely infinite and has real rank zero, Corollary~2.5 in~\cite{LMM}  implies that $\hat T$ is a Jordan epimorphism; hence,
$T$ is a Jordan epimorphism modulo~$J$.
\end{proof}
In particular, if $T$ is even a spectral isometry modulo~$J$ in the above situation, then $\hat T$ provides a Jordan isomorphism
between the quotients $A/I$ and $B/J$. This was obtained in the special case $A=B=B(H)$ and $I=J=K(H)$ for an infinite dimensional
Hilbert space~$H$ in \cite{BBS08}, Theorem~3.1.

In order to make use of quotients, invariant ideals are needed. In the setting of \C*s, so-called Glimm ideals offer themselves as
good candidates, see Section~\ref{sec:spec-isoms} below. However, so far their invariance has only been established under additional
hypotheses, for instance for \spisoms/ on von Neumann algebras~\cite{MMSou}.
\begin{problem}
Does every unital \spisom/ from a unital \C* onto another one leave each Glimm ideal invariant?
\end{problem}
The setting of \C*s is favourable also because a \spisom/ induces a \spisom/ on every quotient (\cite{MMSou}, Proposition~9).
In a more general setting it is easy to see that a \spbddop/ induces a \spbdd/ quotient operator provided the domain algebra
is an \textsl{SR}-algebra; cf.~\cite{MMRud}.

\section{The Relevance of the Value at~$1$}\label{sec:value-at-one}

\noindent
A surjective Jordan homomorphism between unital algebras attains the value~$1$~at~$1$.
In this section, we shall discuss the behaviour of an arbitrary \spbddop/ at~$1$ to see how this affects the possible
difference from a Jordan homomorphism. Remembering that \textit{every\/} bounded operator from a space
$C(X)$ is \spbdd/, we have to be careful not to expect too much in a general setting.

The following observation follows directly from our earlier results.

\begin{prop}\label{prop:basic-fact}
Let\/ $T\colon A\to B$ be a \spbddop/ from a unital \C*\/ $A$ onto a unital semisimple Banach algebra~$B$.
For every pair\/ $p,\,q$ of mutually orthogonal properly infinite projections in\/ $A$ we have
\begin{equation}\label{eq:basic-fact}
(Ta)\,(Tb)+(Tb)\,(Ta)=0\qquad(a\in pAp,\ b\in qAq).
\end{equation}
\end{prop}
\begin{proof}
By~\cite{MMS2}, Corollary~3.2, $T$ preserves elements of square zero. By \cite{LMM}, Proposition~2.1,
every element in the corners $pAp$ and $qAq$ can be written as a finite sum of elements of square zero
(in $pAp$ and $qAq$, respectively), since both $p$ and $q$ are properly infinite.
The claim thus follows from Lemma~3.3 in~\cite{MMS2}.
\end{proof}

\begin{coroll}\label{cor:t-of-one}
Let\/ $T\colon A\to B$ be a \spbddop/ from a unital \C*\/ $A$ with real rank zero onto a unital semisimple Banach algebra~$B$.
Suppose that every non-zero projection in\/ $A$ is properly infinite. Then\/ $T1$ is an invertible element in the centre of~$B$.
\end{coroll}
\begin{proof}
The basic idea of the argument has been used before in some special cases, see, e.g.,~\cite{CuiHou}.
Let $p\in A$ be a non-trivial projection. Applying the identity in Proposition~\ref{prop:basic-fact} to $a=p$ and $b=q=1-p$ we obtain
\[
(Tp)\,T(1-p)+T(1-p)\,(Tp)=0,
\]
that is, $(Tp)\,(T1)+(T1)\,(Tp)=2\,(Tp)^2$. Upon multiplying this identity first on the left, then on the right by $Tp$ and subtracting the
resulting two identities we obtain $(Tp)^2\,(T1)=(T1)\,(Tp)^2$ for every (non-trivial) projection~$p$.

Let $\{p_1,\ldots,p_n\}$ be an orthogonal family of projections in~$A$. By applying~(\ref{eq:basic-fact}) inductively we find that
$\bigl(T\bigl(\sum_i p_i\bigr)\bigr)^2=\sum_i\bigl(Tp_i\bigr)^2$. Since $A$ has real rank zero, it follows that
$(T1)\,(Ta)^2=(Ta)^2\,(T1)$ for all $a$ in a dense subset of $A_{sa}$ and since $T$ is bounded, thus for all $a\in A_{sa}$.
As
\[
ab+ba=(a+b)^2-a^2-b^2
\]
and
\[
(a+ib)^2=a^2-b^2+i(ab+ba)
\]
for all selfadjoint $a,b\in A$, we conclude that $T1$ commutes with $(Tx)^2$ for every $x\in A$. The surjectivity of $T$ entails
that $T1$ commutes with every square of an element in $B$ but since $2y=(1+y)^2-1-y^2$ for each $y\in B$, we finally obtain that
$T1$ belongs to the centre of~$B$.

Going back to the first paragraph of the proof we therefore have $T(p^2)\,(T1)=(Tp)\,(T1)=(Tp)^2$ for every projection $p\in A$.
Using the same argumentation as above and the fact that $A$ has real rank zero another time we conclude that
$T(x^2)\,(T1)=(Tx)^2$ for all $x\in A$. As $T$ is surjective, it follows that $T1$ must be invertible.
\end{proof}

Whenever $T$ is a \spbddop/ such that $T1$ is an invertible element in the centre of the codomain, the mapping
$\tilde T$ defined by $\tilde Tx=(T1)^{-1}\,Tx$, $x\in A$ is a unital \spbddop/; thus we can apply the results known in this situation.

\begin{thm}\label{thm:nonunital-spbdd}
Let\/ $T\colon A\to B$ be a \spbddop/ from a unital \C*\/ $A$ onto a unital semisimple Banach algebra~$B$.
Suppose that
\begin{enumerate}
\item[\textup{(i)}] $A$ is a properly infinite von Neumann algebra, or
\item[\textup{(ii)}] $A$ is a purely infinite \C* with real rank zero.
\end{enumerate}
Then\/ there is a unique Jordan epimorphism\/ $J\colon A\to B$ such that\/ $Tx=(T1)\,Jx$ for all\/ $x\in A$.
Moreover, $T1$ is a central invertible element in~$B$.
\end{thm}
\begin{proof}
Case~(ii) is immediate from Corollary~\ref{cor:t-of-one}, since every non-zero projection in a purely infinite \C* is properly infinite. Thus we can
define $J$ by $Jx=(T1)^{-1}\,Tx$, $x\in A$, which is a Jordan epimorphism by \cite{LMM}, Corollary~2.5.

The case~(i) will not only need the result in the unital case, \cite{MMS2}, Theorem~3.6, in the same manner as just explained
but also an elaboration of the projection techniques used in the main technical result to obtain the unital case, which is
\cite{MMS2}, Proposition~3.4. Note that, by the proof of Corollary~\ref{cor:t-of-one}, it suffices to show that $T1$ commutes with
$(Tp)^2$ for every projection $p\in A$; once this is verified, $T1$ will be a central invertible element. We shall use the same strategy
as in Proposition~3.4 of~\cite{MMS2}.

Suppose first that both $p$ and $1-p$ are properly infinite. Then the assertion follows immediately from identity~(\ref{eq:basic-fact}).

Next suppose that $p$ is properly infinite but $1-p$ is not. By the Halving Lemma, there is a subprojection $f$ of $p$ such that
$p\sim f\sim p-f$, where $\sim$ denotes Murray--von Neumann equivalence. Hence, all projections $f$, $1-f$, $p-f$ and $1-p+f$
are properly infinite, see the proof of \cite{MMS2}, Proposition~3.4. By our first step we thus have
\[
(T1)\,(Tf)^2=(Tf)^2\,(T1)\quad\text{and}\quad(T1)\,T(p-f)^2=T(p-f)^2\,(T1).
\]
Applying (\ref{eq:basic-fact}) to $f$ and $p-f$ we have $(Tp)^2=T(p-f)^2+(Tf)^2$, wherefore $T1$ commutes with~$(Tp)^2$.

Suppose now that $p$ is infinite but not properly infinite. Then there is a non-trivial central projection $z$ in $A$
such that $zp$ is properly infinite and $(1-z)p$ is finite.
We need the following preliminary observation. Suppose that $q$ is a properly infinite projection in $A$ but $1-q$ is not.
Choosing the subprojection $f$ of $q$ as in the last paragraph we have
\begin{equation*}
\begin{split}
(Tq)\,(T1)+(T1)\,(Tq) &= (Tf)\,(T1)+(T1)\,(Tf)+T(q-f)\,(T1)+(T1)\,T(q-f)\\
                                         &= 2\,(Tf)^2+2\,T(q-f)^2 = 2\,(Tq)^2,
\end{split}
\end{equation*}
where the last two equality signs come from identity~(\ref{eq:basic-fact}).
Multiplying this identity first on the left, then on the right by $T1$ and subtracting the two identities we obtain
\[
(T1)^2\,(Tq)-(Tq)\,(T1)^2=2\,\bigl((T1)\,(Tq)^2-(Tq)^2\,(T1)\bigr)=0,
\]
by the previous paragraph.

To simplify our calculations we will now use the usual commutator notation $[x,y]=xy-yx$.
Since $zp$ is a properly infinite projection, the preliminary observation yields $[(T1)^2, T(zp)]=0$.
By the proof of \cite{MMS2}, Proposition~3.4, the projection $q=(1-z)(1-p)$ is properly infinite as well;
therefore, $[(T1)^2,T((1-z)(1-p))]=0$, too. As a result,
\begin{equation*}
\begin{split}
[(T1)^2,T((1-z)p)] &= [(T1)^2,Tp]-[(T1)^2,T(zp)] = -\, [(T1)^2,T(1-p)]\\
                                    &=-\, [(T1)^2,T(z(1-p))]-[(T1)^2,T((1-z)(1-p))]\\
                                    &=-\, [(T1)^2,Tz]+[(T1)^2,T(zp)]=0,
\end{split}
\end{equation*}
since every non-zero central projection in $A$, in particular~$z$, is properly infinite and thus $[(T1)^2,Tz]=0$ as well.

From identity~(\ref{eq:basic-fact}) we have
\begin{equation}\label{eq:basic-fact-central}
T(za)\,T((1-z)b)+T((1-z)b)\,T(za)=0\qquad(a,b\in A),
\end{equation}
since $1-z$ is non-zero, hence properly infinite. In particular,
\[
(Tz)\,T((1-z)p)+T((1-z)p)\,(Tz)=0.
\]
It follows that
\begin{equation*}
\begin{split}
(Tq)^2 &=T((1-z)(1-p))^2 = \bigl(T(1-z)-T((1-z)p)\bigr)^2\\
              &= T(1-z)^2+T((1-z)p)^2-(T(1-z)\,T((1-z)p)+T((1-z)p)\,T(1-z))\\
              &=T(1-z)^2+T((1-z)p)^2-((T1)\,T((1-z)p)+T((1-z)p)\,(T1)).
\end{split}
\end{equation*}
Combining this with
\begin{equation*}
[(T1)^2,Tw]=T1\,[T1,Tw]+[T1,Tw]\,T1=[T1,T1\,Tw+Tw\,T1]\qquad(w\in A)
\end{equation*}
we find that
\begin{equation*}
\begin{split}
0 &= [T1, (Tq)^2]\\
   &= [T1,T(1-z)^2]+[T1,T((1-z)p)^2] - [T1,(T1)\,T((1-z)p)+T((1-z)p)\,(T1)]\\
   &= 0 + [T1,T((1-z)p)^2] - [(T1)^2,T((1-z)p)],
\end{split}
\end{equation*}
that is, $[T1,T((1-z)p)^2]=[(T1)^2,T((1-z)p]$. The commutator on the right hand side is zero, as we saw above.
Therefore,
\begin{equation*}
[T1, (Tp)^2]=[T1,T(zp)^2]+[T1,T((1-z)p)^2]=0,
\end{equation*}
where we used~(\ref{eq:basic-fact-central}) once again.

Finally suppose that $p$ is finite. Then $p'=1-p$ is infinite. Letting $z'$ be the central projection such that $z'p'$ is properly infinite and
$(1-z')p'$ is finite, we recollect the necessary information from the arguments above.
\begin{enumerate}
\item[(i)] $[T1,(Tp')^2]=0$;
\item[(ii)] $[(T1)^2, T((1-z')p']=[T1, T((1-z')p')^2]=0$;
\item[(iii)] $[(T1)^2,T(z'p')]=[T1,T(z'p')^2]=0$.
\end{enumerate}
Consequently,
\begin{equation*}
[(T1)^2,Tp']=[(T1)^2,T(z'p')]+[(T1)^2,T((1-z')p')]=0.
\end{equation*}
As
\begin{equation*}
\begin{split}
[T1,T(1-p)^2] &=[T1, (T1)^2]+[T1,(Tp)^2]-[T1,(T1)\,(Tp)+(Tp)\,(T1)]\\
                           &=[T1,(Tp)^2]-[(T1)^2,Tp]=[T1,(Tp)^2]+[(T1)^2,T(1-p)],
\end{split}
\end{equation*}
we conclude that
\[
[T1,(Tp)^2]=[T1,(Tp')^2]-[(T1)^2,Tp']=0.
\]

This completes the argument that $[T1,(Tp)^2]=0$ for every projection $p\in A$ and, as explained above, this is sufficient to prove the result.
\end{proof}
\begin{rem}
Evidently, all we need to assume on $T$ in the proof of the above theorem is that $T$ is surjective, bounded and preserves elements of square zero.
Thus, the main result in~\cite{LMM}, Theorem~2.4, extends appropriately to the non-unital setting.
\end{rem}

As pointed out above, we cannot expect a result like Theorem~\ref{thm:nonunital-spbdd} for arbitrary \C*s of real rank zero or even von Neumann
algebras. However, in a more restricted setting the situation might change.

\begin{problem}\label{prob:Tone-on-finite-factor}
Let $T$ be a \spbddop/ defined on a finite von Neumann factor (onto a semisimple unital Banach algebra~$B$).
Is $T1$ a non-zero complex multiple of the identity in~$B$?
\end{problem}

\section{Spectral Isometries}\label{sec:spec-isoms}

\noindent
Generalising the classical Banach--Stone theorem, Kadison showed in~\cite{Kad51} that every unital surjective isometry $T$ between two unital
\C*s must be a Jordan *-isomorphism. In fact, he showed a stronger result in \cite{Kad52}, Theorem~2, namely that it suffices that $T$
maps the selfadjoint part $A_{sa}$ of $A$ isometrically onto the selfadjoint part $B_{sa}$ of~$B$.
There are many interesting consequences of these results, all relating isometric properties to selfadjointness. As an example, we state and prove a variant
of a theorem of Chan (\cite{Chan}, Theorem~3).

Recall that the numerical radius $\nu(x)$ of an element $x$ in a unital \C* $A$ is defined by
$\nu(x)=\sup\{|\varphi(x)|\mid\varphi\ \text{a state of }A\}$.

\begin{thm}
Every unital surjective numerical radius-preserving operator\/ $T$ between two unital \C*s $A$ and\/ $B$ is a Jordan *-isomorphism.
\end{thm}
\begin{proof}
Since $\nu$ is a norm, $T$ is injective and it is clear that thus $T^{-1}$ is numerical radius-preserving as well. As $\nu(x)=\|x\|$ for all $x\in A_{sa}$,
by Kadison's theorem it suffices to show that $TA_{sa}\subseteq B_{sa}$. Let $a\in A_{sa}$, $\|a\|=1$. Write $Ta=b+ic$ with $b,c\in B_{sa}$.
Suppose $c\ne0$; then there is $0\ne\gamma\in\sigma(c)$ and we may assume that $\gamma>0$. For each state $\pfi$ on $B$ and each $n\in\NN$,
we have
\begin{equation*}
|\pfi(c+n)|^2 \leq |\pfi(b)|^2+|\pfi(c+n)|^2 = |\pfi(b)+i\pfi(c+n)|^2 = |\pfi(Ta+in)|^2.
\end{equation*}
Hence, for large~$n$,
\begin{equation*}
\nu(a+in)^2=1+n^2<(\gamma+n)^2\leq\nu(c+n)^2\leq\nu(Ta+in)^2
\end{equation*}
which is impossible as $T$ is numerical radius-preserving. Therefore, $c=0$ and $Ta$ is selfadjoint.
\end{proof}

Kadison's theorem in one direction and an application of the Russo--Dye theorem in the other direction show
that a unital surjective spectral isometry is selfadjoint (i.e., maps selfadjoint elements to selfadjoint elements) if and only if
it is an isometry, see \cite{MM05}, Proposition~2.4. These results and others, and
the analogy to Kaplansky's question (Problem~\ref{prob:kaplansky} above),
made us ask the following question which we indeed state as a conjecture in~\cite{MMS1}.

\begin{problem}\label{prob:spiom-conjecture}
Is every unital surjective \spisom/ between unital \C*s a Jordan isomorphism?
\end{problem}

By now, there is a fair number of results affirming this conjecture in reasonable, though not full generality. Combining the methods of
\cite{MM04b}, \cite{MMRud} and~\cite{MMSou} we can cover another new case below.

\begin{thm}\label{thm:no-tracial-states}
Let\/ $T\colon A\to B$ be a unital spectral isometry from a
unital \C*\/ $A$ with real rank zero and without tracial states such that\/
$\prima$ is Hausdorff and totally disconnected onto a unital \C*~$B$. Then\/ $T$ is a Jordan isomorphism.
\end{thm}

The above statement includes in particular the simple case but we will reduce the more general situation to this one.
To this end, we need the notion of a Glimm ideal in a \C*.

There are (at least) five structure spaces associated with a unital \C*~$A$, which we will briefly discuss. Their relation can be depicted as follows.
\[
\irra\longrightarrow\speca\longrightarrow\prima
\,{\buildrel{{\beta}}\over\longrightarrow}\,\glimma\,{\buildrel{{\cong}}\over\longrightarrow}\,\specza
\]
The set $\irra$ of all irreducible representations of $A$ maps onto the set $\speca$ of all unitary equivalence classes of such representations.
Since equivalent representations have the same kernel, there is a canonical surjection from $\speca$ onto $\prima$, the set of all primitive
ideals of~$A$. The latter carries a natural topology, the Jacobson topology, which can be pulled back to $\speca$, and then to $\irra$, to turn
them into topological spaces in a canonical way. Since $Z(A)$ is a commutative unital \C*, $\specza$ allows several equivalent descriptions
of which we choose the one via maximal ideals. For each maximal ideal $M$ of $Z(A)$, the closed ideal
\[
AM=\Bigl\{\sum_{j=1}^n\,x_jz_j\mid x_j\in A,\,z_j\in M,\,n\in\NN\Bigr\}
\]
is called a \textit{Glimm ideal\/} of~$A$. For each $I\in\glimma$, $M=I\cap Z(A)$ gives the generating maximal ideal in $Z(A)$ back,
wherefore there is a bijection between the two sets. Transporting the natural topology of $\specza$ over to $\glimma$ thus turns the latter into
a compact Hausdorff space. From the other direction, we can apply the \textit{complete regularisation map\/}
$\beta\colon\prima\to\glimma$ defined by
\[
P,Q\in\prima\colon\quad P\sim Q \text{ \ if \ } f(P)=f(Q)\qquad\bigl(f\in C_b(\prima)\bigr)
\]
and $\beta(P)$ is the equivalence class with respect to this relation. If $I$ is the Glimm ideal given by $I=A(P\cap Z(A))$ then $\beta(P)$ can be
identified with $\bigcap\limits_{Q\supseteq I}Q$. Among the consequences of this is
$\bigcap\glimma=\bigcap\prima=\{0\}$, that is, the Glimm ideals separate the points of~$A$.

It follows from the definition of the Jacobson's topology that $\prima$ is T$_1$ if and only if every primitive ideal is maximal.
In a similar vein, $\prima$ is T$_2$, i.e., Hausdorff, if and only if every Glimm ideal is maximal (\cite{MMRud}, Lemma~9).

\medskip\noindent
\textbf{Proof of Theorem~\ref{thm:no-tracial-states}.}\quad
\textrm{By hypothesis, $\prima$ is Hausdorff and thus the mapping $\beta$ is a homeomorphism. Consequently, $C(\prima)=C(\glimma)=Z(A)$
has real rank zero; Lemma~8 in \cite{MMRud} therefore yields that $J=TI$ is a Glimm ideal in $B$ for every $I\in\glimma$.
The induced unital surjective operator $\hat T\colon A/I\to B/J$ given by $\hat T(x+I)=Tx+J$, $x\in A$ is a spectral isometry by
\cite{MMSou}, Proposition~9. Since $\prima$ is Hausdorff, every Glimm ideal is maximal so that $A/I$ is simple for each $I\in\glimma$.
Moreover, $A/I$ has real rank zero and no tracial states. Theorem~3.1 in~\cite{MM04b} therefore entails that $\hat T$ is a Jordan isomorphism,
and since this holds for every~$I$, we obtain that $T$ itself is a Jordan isomorphism.}
\hfill$\square$

\medskip
In contrast to results like Theorem~\ref{thm:no-tracial-states}, very little is known about the behaviour of \spisoms/ on \C*s carrying a trace.
In the remainder of this paper we suggest a possible new route to tackle this situation.

The following example is mentioned with less detail in~\cite{Sem98}.

\begin{exam}\label{exam:spbdd-mnc}
Let $T\colon M_n(\CC)\to M_n(\CC)$ be a linear mapping.
Then $T$ is unital, surjective and \spbdd/ if and only if there are a Jordan
automorphism $S$ of $M_n(\CC)$ and a non-zero complex number $\gamma$ such that
\begin{equation}\label{equa:canonical-form}
Tx=\gamma\,Sx+(1-\gamma)\,\tau(x)\qquad\bigl(x\in M_n(\CC)\bigr),
\end{equation}
where $\tau$ denotes the normalised centre-valued trace on~$M_n(\CC)$.

Evidently, the formula~(\ref{equa:canonical-form}) defines a unital mapping which is
\spbdd/, since the images under $S$ and $\tau$ commute. Moreover, $T$ is surjective:
let $y\in\mnc$ and $z\in\mnc$ be such that $Sz=y$. Put $x_1=\smash{\frac1\gamma}(z-\tau(z))$
and $x_2=\tau(z)$. Then $x_1\in\ker\tau$. Setting $x=x_1+x_2$ we have
\begin{equation*}
\begin{split}
Tx
&= T(x_1+x_2)=\gamma Sx_1+\gamma Sx_2+(1-\gamma)\tau(x_1+x_2)\\
&=Sz-\tau(z)+\gamma\tau(z)+(1-\gamma)\tau(z)\\
&=Sz=y.
\end{split}
\end{equation*}

Conversely, if $T$ is spectrally bounded, it leaves $\ker\tau$ invariant, as this space is spanned
by the nilpotent matrices and $T$ preserves nilpotency (\cite{MMS2}, Lemma~3.1). Assuming that $T$ is surjective,
it is in fact bijective and hence remains injective when restricted to $\ker\tau$. Since the
latter is finite dimensional, $\rest{T}{\ker\tau}$ is bijective from $\ker\tau$ to
$\ker\tau$. By \cite{BPW}, there exist a Jordan automorphism $S$ of $\mnc$ and a
non-zero complex number $\gamma$ such that $\rest{T}{\ker\tau}=\gamma\,\rest{S}{\ker\tau}$.
Hence, for each $x\in\mnc$,
\[
T(x-\tau(x))=\gamma\,S(x-\tau(x))
\]
which is nothing but identity~(\ref{equa:canonical-form}), if $T$ is unital.
\end{exam}

Specialising the above description of \spbdd/ operators on matrix algebras to
\spisoms/ we recover Aupetit's result from \cite{Aup93}, Proposition~2, which was proved using
holomorphic methods.

\begin{exam}\label{exam:spisom-mnc}
Every unital \spisom/ $T$ from $\mnc$ into itself is a Jordan automorphism.

Since $T\colon\mnc\to\mnc$ is injective, by \cite{MMS1}, Proposition~4.2, it is
surjective as well. By the above description~(\ref{equa:canonical-form}),
$\rest{T}{\ker\tau}=\gamma\,\rest{S}{\ker\tau}$ for a Jordan automorphism $S$ of~$\mnc$
and $\gamma\in\CC\setminus\{0\}$. Suppose first that $n\geq3$ and
take $y\in\ker\tau$ with $\sigma(y)=\{1,-\frac12\}$; for instance,
$y=e_{11}-\frac12(e_{22}+e_{33})$, where
$e_{ij}$, $1\leq i,j\leq n$ denote the standard matrix units.
Then
\[
\sigma(Ty)=\gamma\,\sigma(Sy)=\gamma\,\sigma(y)
          =\{\textstyle{\gamma,-\frac{\gamma}{2}}\}.
\]
On the other hand, $T$ preserves the peripheral spectrum, by Proposition~4.7 in~\cite{MMS1}.
Therefore $\gamma=1$, and the claim follows from~(\ref{equa:canonical-form}).
In the case $n=2$, note that identity~(\ref{equa:canonical-form}) entails that
$\tau(Tx)=\tau(x)$ for all $x\in\mnc$. Since $T$ always preserves one eigenvalue
(in the peripheral spectrum), it follows that $T$ preserves
the entire spectrum and hence must be a Jordan automorphism.
The case $n=1$ is trivial.
\end{exam}

Let us now state a problem motivated by these examples.

\begin{problem}\label{prob:spsiom-II1f}
Let $A$ be a \II1f/ with normalised centre-valued trace~$\tau$.
Let $T\colon A\to A$ be a unital surjective \spbddop/.
Are there a Jordan automorphism $S$ of $A$ and a non-zero complex number $\gamma$ such that
\begin{equation}\label{equa:canonical-form-factor}
Tx=\gamma\,Sx+(1-\gamma)\,\tau(x)\qquad\bigl(x\in A\bigr)\,?
\end{equation}
\end{problem}

\begin{rems}\quad
1.\ Every Jordan epimorphism $S\colon A\to A$ is either multiplicative or
anti-multiplicative, by a classical result due to Herstein, as $A$ is simple,
hence prime. Therefore its kernel is an ideal of~$A$. Since $A$ is simple, $S$ must be injective.

\smallskip\noindent
2.\ By the remark just made together with the final part of the argument in
Example~\ref{exam:spbdd-mnc} it suffices to find a Jordan epimorphism $S$ on $A$
such that $\rest{T}{\ker\tau}=\gamma\,\rest{S}{\ker\tau}$ for some non-zero
$\gamma\in\CC$.

\smallskip\noindent
3.\ The ``if''-part in Example~\ref{exam:spbdd-mnc} remains valid in general,
so identity~(\ref{equa:canonical-form-factor}) would in fact be a
\textit{characterization\/} of unital surjective \spbddop/s on \II1f/s.

\smallskip\noindent
4.\ It is easily seen that the representation (\ref{equa:canonical-form-factor})
is unique. Suppose that
\[
T=\gamma\,Sx+(1-\gamma)\,\tau(x)= \gamma'\,S'x+(1-\gamma')\,\tau(x)\qquad\bigl(x\in A\bigr)
\]
are two representations of the \spbddop/~$T$, where $\gamma,\gamma'\in\CC\setminus\{0\}$
and both $S$, $S'$ are Jordan automorphisms of~$A$. It suffices to show that
$\gamma=\gamma'$. Since $\gamma\,\rest{S}{\ker\tau}=\gamma'\,\rest{S'}{\ker\tau}$
and both $S$ and $S'$ are spectral isometries, it follows that $|\gamma|=|\gamma'|$.
Let $p\in A$ be a projection with $\tau(p)=\frac13$ and put $x=p-\tau(p)\in\ker\tau$.
Then $\sigma(x)=\{0,1\}-\frac13=\{-\frac13,\frac23\}$.
As both $S$ and $S'$ preserve the spectrum, it follows that
\[
\textstyle{\{-\frac{\gamma}{3},\frac{2\gamma}{3}\}}
=\gamma\,\sigma(x)=\gamma'\,\sigma(x)
=\textstyle{\{-\frac{\gamma'}{3},\frac{2\gamma'}{3}\}}
\]
and thus $\gamma=\gamma'$.
\end{rems}

Adapting the argument in Example~\ref{exam:spisom-mnc} to the \II1f/ situation,
we can infer the following result from a positive answer to Problem~\ref{prob:spsiom-II1f}.

\begin{coroll}[Under the assumption that Problem~\ref{prob:spsiom-II1f} has a positive answer]\label{cor:type-twoone}
Every unital surjective \spisom/\/ $T$ from a \II1f/\/ $A$ onto itself is a Jordan automorphism.
\end{coroll}

\begin{proof}
Suppose the \spisom/ $T\colon A\to A$ can be written in the form~(\ref{equa:canonical-form-factor}).
Then $\rest{T}{\ker\tau}=\gamma\,\rest{S}{\ker\tau}$ for some non-zero $\gamma\in\CC$.
Since $r(x)=r(Tx)=|\gamma|\,r(Sx)=|\gamma|\,r(x)$ for all $x\in\ker\tau$, it follows that
$|\gamma|=1$. Let $p\in A$ be a projection with $\tau(p)=\frac13$.
Then $\sigma(p-\tau(p))=\{0,1\}-\frac13=\{-\frac13,\frac23\}$.
As in Example~\ref{exam:spisom-mnc},
$\sigma(T(p-\tau(p)))=\{-\frac{\gamma}{3},\frac{2\gamma}{3}\}$.
Since $T$ preserves the peripheral spectrum, we conclude that $\gamma=1$.
\end{proof}
We know that Problem~\ref{prob:spiom-conjecture} has a positive answer whenever the domain is a factor of type different
from~{\rm II}$_1$.

\end{document}